\documentclass[12pt]{amsart}

\usepackage{amsfonts}
\usepackage{amsmath}
\usepackage{amssymb}
\usepackage[margin=1.2in]{geometry}
\usepackage{enumitem}
\usepackage[english]{babel}						
\usepackage[hidelinks]{hyperref} 
\usepackage{mathtools}	
\usepackage[overload]{textcase}						
\usepackage{commath}							
\usepackage{xcolor}

\newtheorem{theorem}{Theorem}

\newtheorem{lemma}{Lemma}

\theoremstyle{definition}

\newtheorem{remark}{Remark}
\newtheorem{example}{Example}

\numberwithin{equation}{section}

\newcommand{\C}{\mathbb{C}}

\renewcommand{\Re}{\operatorname{Re}}
\renewcommand{\Im}{\operatorname{Im}}

\newcommand{\I}{\mathrm{i}}

\newcommand{\e}{\mathrm{e}}

\begin{document}

\title[Absence of remainders]{On the absence of remainders in the Wiener-Ikehara and Ingham-Karamata theorems: a constructive approach}

\author[F. Broucke]{Frederik Broucke}
\thanks{F. Broucke was supported by the Ghent University BOF-grant 01J04017}

\address{Department of Mathematics: Analysis, Logic and Discrete Mathematics\\ Ghent University\\ Krijgslaan 281\\ 9000 Gent\\ Belgium}
\email{fabrouck.broucke@ugent.be}
\email{gregory.debruyne@ugent.be}
\email{jasson.vindas@ugent.be}

\author[G.~Debruyne]{Gregory Debruyne}
\thanks{G.~Debruyne acknowledges support by Postdoctoral Research Fellowships of the Research Foundation--Flanders and the Belgian American Educational Foundation. The latter one allowed him to do part of this research at the University of Illinois at Urbana-Champaign.}

\author[J. Vindas]{Jasson Vindas}
\thanks {J. Vindas was partly supported by Ghent University through the BOF-grant 01J04017 and by the Research Foundation--Flanders through the FWO-grant 1510119N}

\subjclass[2010]{11M45, 40E05, 44A10.}
\keywords{Complex Tauberians; Tauberian theorems with remainder; Wiener-Ikehara theorem; Ingham-Karamata theorem; Laplace transform}

\begin{abstract}We construct explicit counterexamples that show that it is impossible to get any remainder, other than the classical ones, in the Wiener-Ikehara theorem and the Ingham-Karamata theorem under just an additional analytic continuation hypothesis to a half-plane (or even to the whole complex plane). 
\end{abstract}

\maketitle

\section{Introduction}
The Wiener-Ikehara theorem and the Ingham-Karamata theorem are two cornerstones of complex Tauberian theory. Both results have numerous applications in diverse areas such as number theory, operator theory, and partial differential equations. We refer to the monographs \cite{A-B-H-N,korevaarbook,Tenenbaumbook} for accounts on these theorems and related complex Tauberian theorems.

The classical Wiener-Ikehara theorem  states that if a function $S$ is non-decreasing on $[0,\infty)$ and has convergent Laplace-Stieltjes transform on the half-plane $\Re s>1$ such that
\begin{equation}
\label{eq: Laplace extension W-I}
\mathcal{L}\{\dif S; s\}-\frac{a}{s-1}= \int_{0}^{\infty} \e^{-sx}\dif S(x) -\frac{a}{s-1}
\end{equation}
admits an analytic extension beyond $\Re s=1$, then $S$ has asymptotic behavior

\begin{equation}
\label{eq: asymp S W-I}
S(x)=a\e^{x}+o(\e^x).
\end{equation}
On the other hand, one version of the Ingham-Karamata theorem says that if a function $\tau$ is Lipschitz continuous on $[0,\infty)$
and if its Laplace transform 
\[
\mathcal{L}\{\tau ; s\}=\int_{0}^{\infty} \tau(x)\e^{-s x}\dif x
\]
has a analytic continuation across the imaginary axis, then
\begin{equation}
\label{eq: asymp tau I-K}
\tau(x)=o(1).
\end{equation}
We have stated here the simplest forms of these results, but we point out that both theorems have been extensively studied over the last century and have been generalized in a variety of ways. For instance, see \cite{B-B-T2016, Chill-Seifert2016, d-vW-I2016, d-vOptIngham, d-vCT, revesz-roton,S,zhang2019} for recent contributions.

In a recent article \cite{d-vAbsenceI} the  last two named authors have proved that, in general, it is impossible to improve the error terms of the asymptotic formulas \eqref{eq: asymp S W-I} and \eqref{eq: asymp tau I-K} in the Wiener-Ikehara theorem and the Ingham-Karamata theorem if one just augments the assumptions of these theorems by asking an additional analytic continuation hypothesis to a half-plane containing $\Re s>1$ or $\Re s>0$, respectively. In the case of the Wiener-Ikehara theorem, this disproves a conjecture by M\"{u}ger \cite{muger}, who had conjectured that a certain remainder could be obtained if \eqref{eq: Laplace extension W-I} can be analytically extended to $\Re s>\alpha$ with some $0<\alpha<1$. It has indeed been shown in \cite{d-vAbsenceI} that no stronger remainder than the one in \eqref{eq: asymp S W-I} can be achieved if  this extra assumption is solely made together with the classical hypotheses.

The proofs of the quoted results on the absence of remainders  in the Wiener-Ikehara and Ingham-Karamata theorems given in  \cite{d-vAbsenceI} are non-constructive as they rely on abstract functional analysis arguments (the open mapping theorem for Fr\'{e}chet spaces). In particular, they do not deliver any concrete counterexample for specific remainders. One might still wonder how such counterexamples could explicitly be found. The goal of this article is to address the latter constructive problem. In fact, we shall construct explicit instances of functions that show the ensuing theorem. Note that Theorem \ref{th: absence in W-I and I-K} improves \cite[Theorem 3.1 and Theorem 4.2]{d-vAbsenceI}. We recall that the notation $f(x)=\Omega_{\pm}(g(x))$ means that there is $c>0$ such that the inequalities $f(x)>cg(x)$ and $f(x)<-cg(x)$ hold infinitely often for arbitrary large values of $x$; in addition, we will often employ Vinogradov's notation $f(t)\ll g(t)$ to denote the asymptotic bound $f(t)=O(g(t))$.

\begin{theorem}
\label{th: absence in W-I and I-K}
Let $\rho$ be a positive function tending to $0$. 
\begin{enumerate}[label=(\roman*)]
\item There is a non-decreasing function $S$ on $[0,\infty)$ whose Laplace-Stieltjes transform
converges for $\Re s>1$ and
\[ \mathcal{L}\{\dif S; s\} - \frac{1}{s-1}
\]
extends to the whole complex plane $\mathbb{C}$ as an entire function, but satisfying the oscillation estimate
\[ S(x)= \e^{x}+\Omega_{\pm}(\rho(x)\e^{x}).
\]
\item There is a smooth function $\tau$ on $(0,\infty)$ with bounded derivative whose Laplace transform
$\mathcal{L}\{\tau; s\}$ can be analytically continued to the whole $\mathbb{C}$, but satisfying the oscillation estimate
\[\tau(x)=\Omega_{\pm}(\rho(x)).
\]
\end{enumerate}

\end{theorem}

We end this introduction by mentioning that it is actually possible to obtain quantified error terms in complex Tauberian theorems for the Laplace transform, but, as e.g. Theorem \ref{th: absence in W-I and I-K} shows, additional assumptions on the Laplace transform besides analytic continuation are required. Determining such conditions is a central problem in modern complex Tauberian theory and much progress on this question has been made in the last decade, see e.g. \cite{B-B-T2016,B-T2010,Chill-Seifert2016,D-S2019,R-S-S,S}. Many of such results are motivated by the theory of operator semigroups  and have applications in partial differential equations and dynamical systems.


\section{Main  constructions}
\label{Section: constructions examples W-I and I-K}
Our construction relies on three main lemmas, which are presented in this section. The first result allows one to regularize functions that increase to infinity slower than $\sqrt{x}$ .

\begin{lemma}
\label{lem: reg}
Let $\omega$ be a positive non-decreasing function on $(0,\infty)$ satisfying
\[
	\lim_{x\to \infty} \omega(x) = \infty \quad \text{and} \quad \omega(x) \ll \sqrt{x}. 
\]
Then there exists $W\in C^{\infty}(0,\infty)$ with the following properties:
\begin{enumerate}[label=(\alph*)]
	\item \label{itm: a} $\omega(x) \ll W(x) \ll \omega(x^{2})$; 
	\item \label{itm: b} $W(ax) \geq aW(x)$ for every $a\leq 1$;
	\item \label{itm: c} $W'(x) \geq 0$; 
	\item \label{itm: d} for any $n\geq 1$ and $x>0$, 
		\[\abs[1]{W^{(n)}(x)} \leq n!\frac{W(x)}{x^n}.\]
\end{enumerate}
\end{lemma}

\begin{proof}
Consider the Poisson kernel of the real line,
\[
	P(x,y) = \frac{y}{y^{2}+x^{2}}=-\Im (z^{-1}), \qquad \text{with }z=x+\I y. \]
We set
\[
	W(y) = \int_{0}^{\infty}\omega(x)P(x,y)\dif x = \int_{0}^{\infty}\omega(xy)P(x,1)\dif x.
\]

We have 
\[ 
	W(y) \geq \int_{1}^{\infty}\omega(xy)P(x,1)\dif x \gg \omega(y);
\]
and 
\begin{align*}
W(y) 	&= \int_{0}^{y}\omega(xy)P(x,1)\dif x + \int_{y}^{\infty}\omega(xy)P(x,1)\dif x \\
		&\ll \omega(y^{2}) + \sqrt{y}\int_{y}^{\infty}\frac{\sqrt{x}}{1+x^{2}}\dif x \ll \omega(y^{2})+1.
\end{align*}
This proves \ref{itm: a}. Property \ref{itm: b} follows immediately from the definition of $W$. For \ref{itm: c}, 
\[
	\pd{P}{y}(x,y) = \frac{x^{2}-y^{2}}{(x^{2}+y^{2})^{2}}, 
\]
so 
\begin{align*}
W'(y) 	&= \int_{0}^{y}\omega(x)\frac{x^{2}-y^{2}}{(x^{2}+y^{2})^{2}}\dif x + \int_{y}^{\infty}\omega(x)\frac{x^{2}-y^{2}}{(x^{2}+y^{2})^{2}}\dif x \\
		&\geq \omega(y)\int_{0}^{\infty}\frac{x^{2}-y^{2}}{(x^{2}+y^{2})^{2}}\dif x = 0.
\end{align*}
Finally, differentiating the second expression for $P$ with respect to $y$, we obtain the bounds
\[
\abs{\dpd[n]{P}{y}(x,y)}= \abs{\Im \dpd[n]{}{y}\left(\frac{1}{z}\right)} = \abs{\Im \left( \I^{n}\frac{d^{n}}{dz^{n}}\left(\frac{1}{z}\right)\right)}= \abs{\Im \left( \frac{(-\I)^{n}n!}{z^{n+1}}\right)}\leq \frac{n!}{|z|^{n+1}};
\]
therefore,
\begin{align*}
\abs[1]{W^{(n)}(y)} 	&\leq n! \int_{0}^{\infty}\frac{\omega(x)}{(x^{2}+y^{2})^{(n+1)/2}}\dif x \\
				&= n!y^{-n} \int_{0}^{\infty}\frac{\omega(xy)}{(x^{2}+1)^{(n+1)/2}}\dif x \\
				&\leq n! y^{-n}W(y).
\end{align*}

\end{proof}

It should be noted that property \ref{itm: d} always yields property \ref{itm: b}.

The rest of this section is devoted to studying properties of various functions associated to the oscillatory function $\cos (x W(x))$, where $W$ satisfies the above properties \ref{itm: c} and \ref{itm: d}.

\begin{lemma}
\label{lem: omega result} Let $W$ be a smooth function tending to $\infty$ which satisfies the properties \ref{itm: c} and \ref{itm: d} stated in Lemma \ref{lem: reg}. Define
\begin{equation}
	\label{eq: def T}
	T(x) 	= \int_{0}^{x}\e^{u}\cos \bigl(uW(u) \bigr) \dif u
\end{equation}
and \begin{equation}
\label{eq: def V}	
	V(x) = W(x) + xW'(x).
\end{equation}
Then,
\begin{equation}
\label{eq: asymptotics T}
T(x) = \frac{\e^{x}}{V(x)}\sin\bigl( xW(x)\bigr) + O\biggl( \frac{\e^{x}}{V(x)^{2}}  \biggr).
\end{equation}

\end{lemma}

\begin{proof}
Integrating by parts,
\begin{align*}
	T(x) 	&
	=  \int_{0}^{x}\frac{\e^{u}}{V(u)}\left(\sin \bigl(uW(u) \bigr)\right)' \dif u
	\\	
	&
	= \frac{\e^{x}}{V(x)}\sin\bigl( xW(x)\bigr) + O(1) 
		 - \int_{1}^{x}\e^{u}\sin\bigl(uW(u)\bigr)\biggl( \frac{1}{V(u)} - \frac{V'(u)}{V(u)^{2}}\biggr)\dif u.
\end{align*}
To estimate the remaining integral, we perform once more integration by parts and obtain that it equals
\[
	\biggl(\frac{1}{V(x)^{2}} - \frac{V'(x)}{V(x)^{3}}\biggr)\e^{x}\cos\bigl(xW(x)\bigr) + O(1)  + O\Biggl( \int_{1}^{x}\abs[3]{ \biggl(\frac{\e^{u}}{V(u)^{2}} - \frac{\e^{u}V'(u)}{V(u)^{3}}\biggr)' } \dif u\Biggr).
\]
The first term is of the desired order of growth in view of the regularity assumption \ref{itm: d}. The derivative inside the integral equals
\[
	\e^{u}\left(\frac{1}{V(u)^{2}} - 3\frac{V'(u)}{V(u)^{3}} - \frac{V''(u)}{V(u)^{3}} + 3\frac{V'(u)^{2}}{V(u)^{4}}\right)=  \frac{\e^{u}}{V(u)^{2}}+O\left(\frac{\e^{u}}{uV(u)^{2}}\right) ,
\]
again by the regularity assumption \ref{itm: d}, and it is thus eventually positive. Hence the integral is bounded by 
\[
	O(1) + \frac{\e^{x}}{V(x)^{2}} - \frac{\e^{x}V'(x)}{V(x)^{3}}.
\]
It remains to observe that property \ref{itm: d} yields  $W(x)\ll x$, which implies that the $O(1)$ terms above are in fact $O(\e^{x}/V(x)^{2})$.
This concludes the proof of the lemma.
\end{proof}

The last key ingredient in our argument is the analytic continuation property of the Laplace transform of $\cos(x W(x))$ that is obtained in the ensuing lemma. Before we state it, let us point out that we use below the bound $W(x)\ll x$. 
\begin{lemma}
\label{lem: analytic continuation}
Suppose $W$ is a smooth function tending to $\infty$ and satisfying \ref{itm: c} and \ref{itm: d} from Lemma \ref{lem: reg}. Then, the Laplace transform 
\[
	\mathcal{L} \{ \cos(xW(x)); s\} = \int_{0}^{\infty}\cos(xW(x))\e^{-sx}\dif x
\]
admits an analytic continuation to the whole complex plane.
\end{lemma}

\begin{proof}
We shall prove the continuation of 
\[	
	F(s) \coloneqq  \int_{0}^{\infty}\e^{\I xW(x)}\e^{-sx}\dif x,
\]
whence the lemma follows since $\mathcal{L} \{ \cos(xW(x)); s\} = \bigl( F(s) + \overline{F(\overline{s})} \bigr)/2$.

Using property \ref{itm: d}, one sees that the $n$-th Taylor coefficient of $W$ at $x$, $c_{n,x}$, satisfies $\abs[0]{c_{n, x}} \leq x^{-n}W(x)$, so that its Taylor series at $x$ has radius of convergence at least $x$. This shows that $W(z)$ has analytic continuation to the half-plane $\Re z >0$. 
The idea of the proof is to shift the integration contour to one where the real part of $\I zW(z)$ is sufficiently negative, in order to obtain an integral which is convergent for any value of $s\in \C$. 

Consider $z = R\e^{\I\theta}$ with $0\leq \theta \leq \pi/5
$. First we deduce some bounds on 
\[
	\Re(\I zW(z))= -R\bigl(\sin \theta \Re W(z) + \cos \theta \Im W(z) \bigr).
\]
Expanding $W$ in its Taylor series around $R\cos \theta$, we get
\begin{align*}
W(R\e^{\I\theta}) 	&= W(R\cos\theta) + \sum_{n=1}^{\infty}(-1)^{n}c_{2n, R\cos\theta}(R\sin\theta)^{2n} \\
				&{}+ \I\sum_{n=0}^{\infty}(-1)^{n}c_{2n+1,R\cos\theta}(R\sin\theta)^{2n+1}.
\end{align*}
Employing the bounds on $c_{n, R\cos\theta}$ and property \ref{itm: c}, which implies $c_{1, R\cos\theta}\geq 0$, we get
\begin{align}
\label{eq: W real part}
\Re W(R\e^{\I\theta})	&\geq W(R\cos\theta) - W(R\cos\theta)(\tan\theta)^{2}\sum_{n=0}^{\infty}(\tan\theta)^{2n},
 \\
\nonumber
\Im  W(R\e^{\I\theta})	&\geq - W(R\cos\theta)(\tan\theta)^{3}\sum_{n=0}^{\infty}(\tan\theta)^{2n}. 
\end{align}
If we  choose $\theta$ such that $(\tan\theta)^{2}\leq 1/W(R)$, we obtain
\begin{equation}
\label{eq: W real part special}
	\Re (\I zW(z)) \leq -R\bigl((\sin\theta) W(R\cos\theta) + O(1)\bigr).
\end{equation}

Consider now the contours
\[
	\Gamma_{R}\colon [R_{0}, R] \to \C\colon r \mapsto r\exp\biggl(\I \arctan\frac{1}{\sqrt{W(r)}}\biggr)
\]
for some $R_{0}$ sufficiently large (so that $\arctan(W(R_{0})^{-1/2})<\pi/5
$), and 
\[
	C_{R}\colon \bigl[0, \arctan \bigl(W(R)^{-1/2}\bigr)\bigr] \to \C\colon \theta \mapsto R\e^{\I\theta}.
\]
Using \eqref{eq: W real part}, one verifies that for $\Re s \geq\sigma_{0}$, with sufficiently large $\sigma_{0}$, the integral of the function $\e^{\I z W(z)}\e^{-s z}$ over $C_{R}$
 tends to 0 as $R\to\infty$. For the integral over $\Gamma_{\infty}$, we employ \eqref{eq: W real part special} and get
\[
	\abs[3]{\int_{\Gamma_{\infty}}\e^{\I zW(z)}\e^{-sz}\dif z} \ll \int_{R_{0}}^{\infty}\exp\biggl(-r\frac{W(r)}{2\sqrt{1+W(r)}} + (C+ \abs{s})r\biggr)\dif r,
\]
for some constant $C$, since $\sin\arctan\bigl(W(r)^{-1/2}\bigr) = (1+W(r))^{-1/2}$, $\dif z = O(1)\dif r$ by property \ref{itm: d}, and $W(r\cos\theta) \geq W(r)/2$ for $\theta\le \pi/3$ by property \ref{itm: b}. Since $\sqrt{W(r)} \to \infty$, the integral over $\Gamma_{\infty}$ converges absolutely and uniformly for $s$ on any compact subset of $\C$, and hence represents an entire function. In conclusion, the formula
\[
	F(s)  = \int_{[0,R_{0}]\cup C_{R_{0}}\cup \Gamma_{\infty}}\e^{\I zW(z)}\e^{-sz}\dif z,
\]
valid for $s$ in a certain right half-plane in view of Cauchy's theorem,
yields the analytic continuation of $F(s)$ to $\C$.
\end{proof}
\begin{remark}
\label{rk: Bloch function} Let $W$ be an unbounded smooth function satisfying properties \ref{itm: c} and \ref{itm: d} from Lemma \ref{lem: reg}. Similarly as in Lemma \ref{lem: omega result},
\[ \int_{0}^{x}\e^{\sigma u}\cos \bigl(uW(u) \bigr) \dif u=
 \frac{\e^{\sigma x}}{V(x)}\sin\bigl( xW(x)\bigr) + O_{\sigma}\biggl( \frac{\e^{\sigma x}}{V(x)^{2}}  \biggr)
\]
is unbounded for each fixed $\sigma>0$. This and Lemma \ref{lem: analytic continuation} imply that the function $f(x)=\e^{\sigma_{0}x}\cos \bigl(xW(x) \bigr)$ furnishes an example of an exponentially bounded function with abscissa of convergence $\sigma_{0}$ for its Laplace transform (as an improper integral), but whose Laplace transform has entire extension. Furthermore, one might verify that this entire extension is unbounded on any half-plane $\Re s>\sigma_1$ with $\sigma_1<\sigma_0$. Interestingly, Bloch \cite{Bloch1949} has given an example of a function whose Laplace transform extends to an entire function that is bounded on every right half-plane, but has finite abscissa of convergence. In contrast to our example, these properties imply that Bloch's function cannot be exponentially bounded, as follows from \cite[Theorem~4.4.19, p.~287]{A-B-H-N}, but this can also be readily seen from Bloch's construction.
\end{remark}


\section{The examples }
\label{section: the examples W-I and I-K}
We have already done all the necessary work in order to establish Theorem \ref{th: absence in W-I and I-K}. We set 
\[
\tilde{\rho}(x) = \sup_{y\ge x}\rho(y), \quad \omega(x) = \min\bigl(\sqrt{x}, 1/\tilde{\rho}\bigl(\sqrt{x}\bigr)\bigr),
\]
and let $W$ then be a function fulfilling the conditions \ref{itm: a}-\ref{itm: d} from Lemma \ref{lem: reg}.

\begin{example}[Proof of Theorem \ref{th: absence in W-I and I-K}(i)]
\label{counterexample 1 W-I}
We consider the non-decreasing function
\[
	S(x) = \int_{0}^{x}\e^{u}\bigl(1 + \cos\bigl(uW(u) \bigr)\bigr) \dif u, \qquad x\geq 0.
\]
 Since 
\begin{align*}
W(x)\leq V(x) \leq 2 W(x) \ll \omega(x^{2}) \leq  1/\rho(x),
\end{align*}
Lemma \ref{lem: omega result} tells us that
 $S(x) = \e^{x} + \Omega_{\pm}(\e^{x}\rho(x))$. On the other hand, by Lemma \ref{lem: analytic continuation}, its Laplace-Stieltjes transform $\mathcal{L}\{\dif S; s\}$ extends to a meromorphic function on $\C$ with a single simple pole with residue 1 at $s=1$.
\end{example}

\begin{example}[Proof of Theorem \ref{th: absence in W-I and I-K}(ii)]
\label{counterexample 2 I-K}
This time we define our example as
\[
	\tau(x) = \int_{x}^{\infty} \cos\bigl(uW(u) \bigr) \dif u, \qquad x\geq 0.
\]
Then, integrating by parts as in Lemma \ref{lem: omega result}, using  property \ref{itm: d}, the bound $V(x)\asymp W(x)$, and the fact that $W$ is non-decreasing,
\begin{align*}
\int_{x}^{y}\cos\bigl(uW(u) \bigr) \dif u &= \frac{\sin (yW(y))}{V(y)}- \frac{\sin (xW(x))}{V(x)}+ \frac{V'(x)\cos (xW(x))}{V(x)^{3}}\\
&
\quad - \frac{V'(y)\cos (yW(y))}{V(y)^{3}} +\int_{x}^{y} \cos(uW(u))\left(\frac{V''(u)}{V(u)^{3}}-3\frac{V'(u)^{2}}{V(u)^{4}}\right) \dif u
\\
&
=\frac{\sin (yW(y))}{V(y)}- \frac{\sin (xW(x))}{V(x)}+ O\left(\frac{1}{V(x)^{2}}\right) +\int_{x}^{y} O\left(\frac{1}{u^{2}V(u)^{2}}\right) \dif u
\\
&
= \frac{\sin (yW(y))}{V(y)}- \frac{\sin (xW(x))}{V(x)}+ O\left(\frac{1}{V(x)^{2}}\right),
\end{align*}
so that the defining improper integral indeed converges and $\tau(x) = \Omega_{\pm}(1/V(x)) = \Omega_{\pm}(\rho(x))$. That the Laplace transform of $\tau$ has entire extension follows directly from Lemma \ref{lem: analytic continuation}.
\end{example}

\end{document}